\documentclass[12pt,a4paper,reqno]{amsart}

\usepackage{float}

\usepackage{anysize}

\usepackage[utf8]{inputenc}

\usepackage{amsmath}
\usepackage{amssymb}
\usepackage{array}

\usepackage{graphicx}

\usepackage{color}

\newtheorem{observation}{Remark}
\newtheorem{theorem}{Theorem}[section]
\newtheorem{corollary}[theorem]{Corollary}
\newtheorem{lemma}[theorem]{Lemma}

\newtheorem{example}[theorem]{Example}

\newtheorem{conjecture}[theorem]{Conjecture}

\begin{document}

\title[Some sums over the non-trivial zeros of $\zeta(s)$]{Some sums over the non-trivial zeros \\ of the Riemann zeta function}
\author{Jesús Guillera}
\address{Avda. Cesáreo Alierta, 31, esc. izda. $4^{o}$-A, Zaragoza 50008, Spain}
\email{jguillera@gmail.com}
\date{}

\maketitle

\begin{abstract}
We prove some identities, which involve the non-trivial zeros of the Riemann zeta function. From them we derive some convergent asymptotic expansions related to the work by Cramér, and also new representations for some arithmetical functions in terms of the non-trivial zeros. 
\end{abstract}

\section{Introduction}
In $1737$ Euler proved that for $s>1$, the series and the product below are convergent and that
\[ 
\sum_{n=1}^{\infty} \frac{1}{n^s}=\prod_p \left( 1-\frac{1}{p^s} \right)^{-1},
\]
showing a connection among the series over the positive integers in the left side and the product over the prime numbers in the right side. In $1859$ Riemann had the bright idea of extending the function to the set of complex numbers in an analytic way, and he called $\zeta(s)$ to this function, which is meromorphic having only a pole at $s=1$. He achieved it by proving the functional equation
\[
\zeta(1-s)=2(2\pi)^{-s} \Gamma(s) \cos \left( \frac{\pi s}{2} \right) \zeta(s).
\]
The function $\zeta(s)$ has trivial zeros at $s=-2n$ for $n=1,2,\dots$. The other zeros $\rho=\beta+i\gamma$ are therefore called the non-trivial zeros and are complex. The following convergent series for $\zeta(s)$, valid for all $s \neq 1$, was proved by Hasse in $1930$ and rediscovered by Jonathan Sondow \cite{sondow}:
\[
\zeta(s)=\frac{1}{1-2^{1-s}} \sum_{n=0}^{\infty} \frac{1}{2^{n+1}}\sum_{k=0}^{n} (-1)^k \binom{n}{k} (k+1)^{-s}.
\]
In $1893$ Hadamard proved that 
\[
\xi(s):=\pi^{-s/2} (s-1) \Gamma \left( 1 + \frac{s}{2} \right) \zeta(s)=\frac12 \prod_{\rho} \left( 1 - \frac{s}{\rho} \right).
\]
Riemann also knew this product and arrived at it using, in Edwards' words \cite[p. 18]{edwards}, an obscure argument. As the product shows that the zeros of zeta determine the whole function, Riemann thought that there had to be a connection among the prime numbers and the zeros of zeta. He showed directly this relation by proving an explicit formula for the number of primes $\pi(x)$ less or equal than $x$ in terms of the complex zeros of $\zeta(s)$. A simpler variant of his formula is 
\[
\psi(x)=\sum_{n \leq x} \Lambda(n)=x - \log 2\pi - \frac12 \left( 1-\frac{1}{x^2} \right)- \sum_{\rho} \frac{x^\rho}{\rho}, 
\]
valid for $x$ not a prime power, where $\Lambda(n)$ is the Mangoldt function, which is defined by $\Lambda(n)=\log p$ if $n$ is a power of $p$, and $0$ otherwise. The asymptotic approximation $\psi(x) \sim x$ is equivalent to the Prime Number Theorem conjectured by Gauss, namely
\[
\pi(x) \sim \frac{x}{\log x}.
\]
From the functional equation of zeta, Riemann proved that all the non-trivial zeros of zeta are in the band $\beta \in [0,1]$ and in $1896$ Hadamard and de La Vallée Poussin achieved to prove the Prime Number Theorem by showing that $\zeta(s)$ has no zeros of the form $\rho=1+i\gamma$. The Riemann's famous conjecture stating that the real part of all the non-trivial zeros was equal to $1/2$ remains unproved. Riemann introduced the notation $\rho=1/2+i \tau$ for the zeros of zeta. In this way his conjecture is the statement that all the $\tau's$ are real. In addition, Riemann conjectured a convergent asymptotic behavior, as $T \to +\infty$, for the number of complex zeros with positive imaginary part less or equal than $T$:
\[
\mathcal{N}(T) \sim \frac{T}{2\pi} \log \frac{T}{2\pi}-\frac{T}{2\pi},
\]
which was proved by von Mangoldt, who improved this result, by showing in $1905$ that
\[
\mathcal{N}(T) = \frac{T}{2\pi} \log \frac{T}{2\pi}-\frac{T}{2\pi} + \mathcal{O}(\log T).
\]
Around $1912$ Backlund derived an exact formula for counting the zeros \cite{edwards}:
\[
\mathcal{N}(T)=\frac{1}{\pi} \theta(T)+\frac{1}{\pi} \arg \zeta \left(\frac12+iT \right)+1,
\]
where $\theta(T)$ is the Riemann-Siegel theta function, which is defined by
\begin{align}
\theta(T) &=\arg \Gamma \left(\frac14+\frac{i}{2}T \right)-\frac{\log \pi}{2} T
\\ \nonumber &= \frac{T}{2} \log \frac{T}{2\pi}-\frac{T}{2}-\frac{\pi}{8}+\frac{1}{48T}+\frac{7}{5760T^3}+\frac{31}{80640T^5}+\cdots. \nonumber
\end{align}
In $1911$ E. Landau proved the following formula for the Mangoldt function \cite{Landau}, \cite{Kalape}:
\[
\Lambda(t)=\frac{-2\pi}{T} \sqrt{t} \sum_{0< {\rm Re} \; \tau \leq T} \cos(\tau \log t)
+\mathcal{O}(\frac{\log T}{T}),
\]
which implies
\[
\Lambda(t)=-2\pi \sqrt{t} \lim_{T \to +\infty} \frac{1}{T}  \sum_{0< {\rm Re} \; \tau \leq T} \cos(\tau \log t).
\]
It has the surprising property that neglecting a finite number of zeros of zeta we still recover the Mangoldt function. Related to it is the self-replicating property of the zeros of zeta observed recently in the statistics of \cite{perez-marco}, and later proved in \cite{ford-zaha}.

\par In this paper we prove a formula which resembles in a certain way to that of Landau:
\[
\Lambda(t) = -4 \pi \sqrt{t} \cot \frac{x}{2} \sum_{{\rm Re} \, \tau>0} \left( \frac{\sinh x \tau}{\sinh \pi \tau} \cos(\tau \log t) \right)+2\pi \cot \frac{x}{2} \left( t-\frac{1}{t^2-1} \right)+o\left( \cot \frac{x}{2} \right),
\]
which implies
\[
\Lambda(t)=-4 \pi \sqrt{t} \lim_{x \to \pi^{-}} \left( \cot \frac{x}{2} \sum_{{\rm Re} \, \tau>0} \frac{\sinh x \tau}{\sinh \pi \tau} \cos(\tau \log t) \right).
\]
We see that it shares with that given by Landau the property of invariance when we neglect a finite number of zeros. However our representation is smooth. In addition we give 
representations, of the same nature, for the functions of Moebius and Euler-Phi. Other results in this paper, are some asymptotic expansions related to Crámer's work but which are more explicit, and other kind of limit evaluations, one of which writes
\begin{align}
\lim_{x \to \pi^{-}} \sum_{\gamma>0} \frac{\sinh x\gamma}{\sinh \pi \gamma} 
\left( \frac{\log 2}{\sqrt{2}} \cos(\gamma \log t)-\frac{\Lambda(t)}{\sqrt{t}}\cos(\gamma \log 2)
\right) \nonumber \\ = \frac{\log 2}{\sqrt{2}} \left( \frac{\sqrt{t}}{2}-\frac{1}{2(t^2-1)\sqrt{t}} \right)-\frac{\Lambda(t)}{\sqrt{t}} \frac{5\sqrt{2}}{12},
\end{align}
if we assume the Riemann Hypothesis.

\section{Series involving the Riemann zeta function}

The formulas that we will prove involve a sum over the non-trivial zeros of zeta. We use the notation $\rho=\beta+i \gamma$ for these non-trivial zeros. Following Riemann, we define $\tau=-i (\rho-\frac12)$. Hence $\rho=1/2+i \tau$ (Riemann used the notation $\alpha$ instead of $\tau$). The Riemann Hypothesis is the statement that all the $\tau's$ are real. We use $C$ for Euler's constant as Euler did. As usual in papers of Number Theory, $\log$ denotes the naperian logarithm.

\subsection{Introduction}

We prove some lemmas that we will need later

\begin{lemma}\label{lema-cota-1}
Let $s=\sigma+it$. Then, for $t>1$, we have
\begin{equation}\label{zsin-1}
\left| \frac{z^s}{\sin \pi s} \right| \leq 4 \, e^{-t(\pi+\arg(z))},
\end{equation}
and for $t<-1$, we have
\begin{equation}\label{zsin-2}
\left| \frac{z^s}{\sin \pi s} \right| \leq 4 \, e^{t(\pi-\arg(z))},
\end{equation}
in the following cases: 
\begin{enumerate}
\item when $\sigma>0$ and $|z|<1$
\item when $\sigma<0$ and $|z|>1$
\end{enumerate}
\end{lemma}
\begin{proof}
We prove the lemma for $t>1$:
\begin{align}
\left| \frac{z^s}{\sin \pi s} \right| &= 2 \left| 
\frac{ e^{(\sigma+it)(\log|z|+i\arg(z))} }{ e^{i\pi(\sigma+it)}-e^{-i\pi(\sigma+it)} } \right| \leq 2 \, \frac{ e^{\sigma \log |z|-t \arg(z)} }{ |e^{-i \pi \sigma} e^{\pi t}|-|e^{i \pi \sigma} e^{-\pi t}| } \nonumber \\ & \leq  2 \, \frac{e^{\sigma \log |z|}e^{-t \arg(z)} }{ e^{\pi t}-e^{-\pi t} } <  4 \, \frac{e^{\sigma \log |z|}e^{-t \arg(z)}}{e^{\pi t}} < 4 \, e^{-t(\pi+\arg(z))}. \nonumber
\end{align}
The proof for $t<-1$ is similar.
\end{proof}

\begin{lemma}\label{lema-cota-2}
Let $s=\sigma+it$, with $\sigma$ being a semi-integer. Then, we have
\begin{equation}\label{cota-S}
\left| \frac{z^s}{\sin \pi s} \right| \leq 2 \, e^{\sigma \log|z|}.
\end{equation}
\begin{proof}
We prove it for $t \geq 0$:
\[
\left| \frac{z^s}{\sin \pi s} \right| = \left| \frac{e^{(\log|z|+i\arg(z))(\sigma+it)}}{\sin \pi(\sigma+it)} \right| = \frac{e^{\sigma \log|z|-t\arg(z)}}{\cosh \pi t} \leq 2 \, e^{\sigma \log|z|} e^{-t(\pi+\arg(z))}.
\]
In a similar way for $t<0$, we get 
\[
\left| \frac{z^s}{\sin \pi s} \right| \leq 2 \, e^{\sigma \log|z|} e^{t(\pi-\arg(z))}.
\]
Both cases imply (\ref{cota-S}).
\end{proof}
\end{lemma}

\begin{lemma}\label{lema-cota-3}
Let us consider the vertices $A=-1/2+iT$, \, $B=R+iT$, \, $C=R-iT$, \, $D=-1/2-iT$, \, $C'=S-iT$, \, $B'=S+iT$, where $S=1/2-2 \lfloor T \rfloor$ and $R=1/2+\lfloor T \rfloor$. If $\phi(s)$ is a function such that
\[
\lim_{T \to +\infty} \frac{|\phi(\sigma \pm iT)|}{e^{c T}}=\lim_{T \to +\infty} \frac{|\phi(R \pm it)|}{e^{c T}}=\lim_{T \to +\infty} \frac {|\phi(S \pm it)|}{e^{cT}}=0 \quad \forall c>0,
\]
then 
\[
\lim_{T \to +\infty} \left( \int_{H} \phi(s) \frac{z^s}{\sin \pi s} ds \right)=0
\]
for $|z|<1$ and $H$ being each of the segments $AB$, $BC$, $CD$, and for $|z|>1$ and $H$ being each of the segments $AB'$, $B'C'$, $C'D$.
\end{lemma}
\begin{proof}
By Lemma \ref{lema-cota-1} the integrals along the segments $AB$, $CD$ for $|z|<1$, and $AB'$, $C'D$ for $|z|>1$ are equal to $0$, and by Lemma \ref{lema-cota-2} the integrals along the segments $BC$ for $|z|<1$ and $B'C'$ for $|z|>1$ are also $0$.
\end{proof}

\begin{lemma}\label{lema-Zeta}
The function $\phi(s)=\zeta'(s+1/2)/\zeta(s+1/2)$ satisfies the hypothesis of Lemma \ref{lema-cota-3} for suitable numbers $T$.
\end{lemma}
\begin{proof}
It is known that for every real number $T' \geq 2$, there exist $T \in [T', T'+1]$ such that uniformetly for $-1 < \sigma \leq 2$, we have \cite[Corollary 3.90]{bordelles}. 
\[
\left| \frac{\zeta'(\sigma+iT)}{\zeta(\sigma+iT)} \right| = \mathcal{O}(\log^2 T).
\]
We will choose numbers $T$ which satisfy the above condition. In addition, suitable bounds for $\sigma \geq 1$ and $T \geq 8$ are known \cite[Theorem 2.4-(d)]{ellison}:
\[ 
\left| \frac{\zeta'(\sigma+iT)}{\zeta(\sigma+iT)} \right| = \mathcal{O}(\log^9 T).
\]
For $\sigma \geq 2$, we have
\[ 
\left| \frac{\zeta'(\sigma+it)}{\zeta(\sigma+it)} \right| = \left| \sum_{n=1}^{\infty} \frac{\Lambda(n)}{n^{\sigma+it}} \right| \leq \sum_{n=1}^{\infty} \left| \frac{\Lambda(n)}{n^{\sigma+it}} \right|
= \sum_{n=1}^{\infty} \frac{\Lambda(n)}{n^{\sigma}} \leq \sum_{n=1}^{\infty} \frac{\Lambda(n)}{n^2} = \mathcal{O}(1),
\]
and for $\sigma < -1$, we know that
\[
\left| \frac{\zeta'(\sigma+iT)}{\zeta(\sigma+iT)} \right|= \mathcal{O}(\log \sqrt{\sigma^2+T^2}+1),
\]
supposing that circles of radius $1/4$ around the trivial zeros $s=-2k$ of $\zeta(s)$ are excluded \cite[Lemma 12.4]{montgomery}. The lemma follows from these bounds. Observe that the circles we need to exclude justify the choice $S=1/2-2 \lfloor T \rfloor$ in Lemma \ref{lema-cota-3}.
\end{proof}

\begin{lemma}\label{lema-reczeta}
If we assume the Riemann hypothesis, then we can prove that the function $\phi(s)=1/\zeta(s+1/2)$ satisfies the hypothesis of Lemma \ref{lema-cota-3}, for suitable numbers $T$.
\end{lemma}
\begin{proof}
If $\varepsilon$ is any positive number, then assuming the RH it is known that for $1/2 \leq \sigma \leq 2$ each interval $[T',T'+1]$ contains a $T$ such that 
\[
\left| \frac{1}{\zeta(\sigma+iT)} \right| = \mathcal{O}(T^{\varepsilon}),
\]
\cite[eq. 14.16.2]{titch}. 
For $\sigma \geq 2$, we have
\[ 
\left| \frac{1}{\zeta(\sigma+it)} \right| = \left| \sum_{n=1}^{\infty} \frac{\mu(n)}{n^{\sigma+it}} \right| \leq \sum_{n=1}^{\infty} \left| \frac{\mu(n)}{n^{\sigma+it}} \right|
\leq \sum_{n=1}^{\infty} \frac{1}{n^{\sigma}} \leq \sum_{n=1}^{\infty} \frac{1}{n^2} = \mathcal{O}(1),
\]
Hence, from \cite[eq. 2.17]{GeMi}, we get suitable bounds for $-1 \leq \sigma \leq 1/2$ and for $\sigma \leq -1$. For $\sigma > 1$ and $T \geq 8$, we know that
\[
\left| \frac{1}{\zeta(\sigma+it)} \right| = \mathcal{O}(\log^7 T),
\]
\cite[Theorem 2.4-(c)]{ellison}. This bounds prove the lemma.
\end{proof}

\subsection{Series with the Mangoldt function}

\begin{theorem}\label{main-thm}
Let $\Omega = \mathbb{C}-(-\infty,0]$ (the plane with a cut along the real negative axis). We shall denote by $\log z$ the main branch of the $\log$ function defined on $\Omega$ taking $|\arg(z)|<\pi$. We also denote by $z^s=\exp(s\log(z))$, the usual branch of $z^s$ defined also on $\Omega$. For all $z \in \Omega$ we have
\begin{equation}\label{for-main-thm}
\sum_{n=1}^{\infty} \frac{\Lambda(n)z}{\pi \sqrt{n} (z+n)}-\sum_{n=1}^{\infty} \frac{\Lambda(n)}{\pi \sqrt{n} (1+nz)}=\sqrt{z}-\frac{\zeta'(\frac12)}{\pi \zeta (\frac12)}-2\sum_{{\rm Re} \, \tau>0} \frac{\sin (\tau \log z) }{\sinh \pi \tau} + h(z),
\end{equation}
where
\begin{equation}\label{h-of-z}
h(z)=\frac{1}{\sqrt{z}(z^2-1)}-\frac{1}{2z-2}+\frac{\log(8\pi)+C}{\pi} \frac{1}{z+1}+\frac{i}{\pi} \frac{\sqrt{z}}{z+1}\log \frac{\sqrt{z}+i}{\sqrt{z}-i}.
\end{equation}
\end{theorem}

\begin{observation}
{\rm This formula is a specialization of \cite[Lemma 1]{Go-Go} taking 
\[ 
h(u)=\frac{\sin(u \log z)}{\sinh(\pi u)}. 
\]
It needs to evaluate
\[
\frac{1}{2\pi} \int_{-\infty}^{\infty} \frac{\sin(u \log z)}{\sinh(\pi u)} {\rm Re} \frac{\Gamma'}{\Gamma} 
\left( \frac14 + \frac{iu}{2} \right) du,
\]
and
\[
\hat{h}(x)=\int_{-\infty}^{\infty} \frac{\sin(u \log z)}{\sinh(\pi u)} e^{-2\pi i x u} du.
\]
However our proof of this particular case is direct and probably simpler.}
\end{observation}

\begin{proof}
Let $I(z)$ be the analytic continuation of the function
\[
I(z)=\frac{1}{2\pi i} \int_{-1/2-i\infty}^{-1/2+i\infty} \frac{\zeta'(s+\frac12)}{\zeta(s+\frac12)} \frac{\pi}{\sin \pi s} z^s ds.
\]
along the vertical axis $x=-1/2$. For evaluating $I$, we first consider rectangles of vertices $A=-1/2+iT$, \, $B=R+iT$, \, $C=R-iT$, \, $D=-1/2-iT$, where $R=1/2+\lfloor T \rfloor$. If we choose suitable real numbers $T$ and use Lemma \ref{lema-Zeta}, we see that for $|z|<1$ the integrals, as $T \to \infty$, along the sides $AB$, $BC$ and $CD$ are equal to zero. By applying the residues theorem, we obtain
\begin{align}
& I={\rm res}_{s=\frac12} \left( \frac{\zeta'(s+\frac12)}{\zeta(s+\frac12)} \frac{\pi}{\sin \pi s} z^s \right)+\sum_{n=0}^{\infty} {\rm res}_{s=n} \left( \frac{\zeta'(s+\frac12)}{\zeta(s+\frac12)} \frac{\pi}{\sin \pi s} z^s \right) +  \nonumber  \\
& \sum_{\rho} {\rm res}_{s=\rho-\frac12} \left( \frac{\zeta'(s+\frac12)}{\zeta(s+\frac12)} \frac{\pi}{\sin \pi s} z^s \right) \!=\! - \pi \sqrt{z}+\sum_{n=0}^{\infty} (-1)^n \frac{\zeta'(n+\frac12)}{\zeta(n+\frac12)}z^n+\pi\sum_{\rho} \frac{z^{\rho-\frac12}}{\sin \pi (\rho-\frac12)}. \nonumber
\end{align}
By analytic continuation, we have that for all $z \in \Omega$
\begin{equation}\label{int-right}
I = - \pi \sqrt{z}+\frac{\zeta'(\frac12)}{\zeta(\frac12)}+\sum_{n=1}^{\infty} \frac{\Lambda(n) z}{\sqrt{n}(z+n)} + \pi \sum_{\rho} \frac{z^{\rho-\frac12}}{\sin \pi (\rho-\frac12)}
\end{equation}
In a similar way, for $|z|>1$ we integrate along the rectangles of vertices $A=-1/2+iT$, \, $B'=S+iT$, \, $C'=S-iT$, \, $D=-1/2-iT$, where $S=1/2-2 \lfloor T \rfloor$. By Lemma \ref{lema-Zeta} we see that when $T \to +\infty$, the integrals along the sides $AB'$, $B'C'$ and $C'D$ are equal to zero. Hence, by the residues method, we deduce that
\begin{align}
I &= \sum_{n=1}^{\infty} {\rm res}_{s=-2n-\frac12} \left( \frac{\zeta'(s+\frac12)}{\zeta(s+\frac12)} \frac{\pi}{\sin \pi s} z^s \right) + \sum_{n=1}^{\infty} {\rm res}_{s=-n} \left( \frac{\zeta'(s+\frac12)}{\zeta(s+\frac12)} \frac{\pi}{\sin \pi s} z^s \right) \nonumber \\
&= \sum_{n=1}^{\infty} \frac{\zeta'(-2n)}{\zeta(-2n)} \sin(2 \pi n) z^{-2n-\frac12}+\sum_{n=1}^{\infty} (-1)^n \frac{\zeta'(\frac12-n)}{\zeta(\frac12-n)}z^{-n}, \label{lastsums}
\end{align}
where we understand the expression inside the first sum of (\ref{lastsums}) as a limit based on the identity
\[ 
\lim_{s \to -2n} (s+2n) \frac{\zeta'(s)}{\zeta(s)} = \lim_{s \to -2n} \frac{\zeta'(s)}{\zeta(s)}  \frac{\pi (s+2n)}{\sin \pi s} \frac{\sin \pi s}{\pi} = \lim_{s \to -2n} \frac{\sin{\pi s}}{\pi} \frac{\zeta'(s)}{\zeta(s)}.
\] 
We use the functional equation (which comes easily from the functional equation of $\zeta(s)$):
\begin{equation}\label{zpz}
\frac{\zeta'(1-s)}{\zeta(1-s)}=\log 2\pi-\psi(s)-\frac{\pi}{2}\tan \frac{\pi s}{2}-\frac{\zeta'(s)}{\zeta(s)}.
\end{equation}
to simplify the sums in (\ref{lastsums}). For the first sum in (\ref{lastsums}), we obtain 
\[
\sum_{n=1}^{\infty} \frac{\zeta'(-2n)}{\zeta(-2n)} \sin(2\pi n) z^{-2n-\frac12}=\pi \sum_{n=1}^{\infty} z^{-2n-\frac12},
\]
and for the last sum in (\ref{lastsums}), we have
\[
-\log 2 \pi \sum_{n=1}^{\infty} (-1)^n z^{-n}+\sum_{n=1}^{\infty} (-1)^n \psi \left(\frac12+n \right)z^{-n}-\frac{\pi}{2}\sum_{n=1}^{\infty} z^{-n}+\sum_{n=1}^{\infty} (-1)^n \frac{\zeta'(n+\frac12)}{\zeta(n+\frac12)}z^{-n},
\]
where $\psi$ is the digamma function, which satisfies the property
\[
\psi \left(\frac12+n \right) = 2 h_n - C - 2 \log 2, \qquad h_n = \sum_{j=1}^n \frac{1}{2j-1}.
\]
Using the identity, due to Hongwei Chen \cite[p.299, exercise 34]{bailey-et-al}
\[
2 \sum_{n=1}^{\infty} (-1)^n h_n z^{-n} = i \frac{\sqrt{z}}{z+1} \log \frac{\sqrt{z}+i} {\sqrt{z}-i},
\]
we get that for $|z|>1$
\[
I = \frac{\pi}{\sqrt{z}(z^2-1)}+\frac{\log 2 \pi}{z+1}+\frac{C+\log 4}{z+1}-\frac{\pi}{2z-2} +  i \frac{\sqrt{z}}{z+1} \log \frac{\sqrt{z}+i} {\sqrt{z}-i}+\sum_{n=1}^{\infty} (-1)^n \frac{\zeta'(n+\frac12)}{\zeta(n+\frac12)}z^{-n}
\]
Then, by analytic continuation, we obtain that for all $z \in \Omega$:
\begin{equation}\label{int-left}
I=\frac{\pi}{\sqrt{z}(z^2-1)}+\frac{C+\log 8\pi}{z+1}-\frac{\pi}{2z-2} +  i \frac{\sqrt{z}}{z+1} \log \frac{\sqrt{z}+i} {\sqrt{z}-i}+\sum_{n=1}^{\infty} \frac{\Lambda(n)}{\sqrt{n}(1+zn)}.
\end{equation}
By identifying (\ref{int-right}) and (\ref{int-left}), and observing that the pole at $z=1$ is removable, we can complete the proof.
\end{proof}

\begin{corollary}
For $x>0$, the function $h(z)$ in (\ref{h-of-z}) becomes
\[
h(x)=\frac{1}{\sqrt{x}(x^2-1)}-\frac{1}{2x-2}+\frac{\log(8\pi)+C}{\pi} \frac{1}{x+1}-\frac{2}{\pi} \frac{\sqrt{x}}{x+1} \arctan \frac{1}{\sqrt{x}}.
\]
\end{corollary}
\begin{proof}
For $x>0$ we have
\[
\sqrt{x}+i=\sqrt{x+1} \, \exp \left(i \arctan \frac{1}{\sqrt{x}} \right), \quad 
\sqrt{x}-i=\sqrt{x+1} \, \exp \left(-i \arctan \frac{1}{\sqrt{x}} \right).
\]
Hence
\[
\frac{i}{\pi} \frac{\sqrt{x}}{x+1} \log \frac{\sqrt{x}+i}{\sqrt{x}-i}=\frac{-2}{\pi} 
\frac{\sqrt{x}}{x+1} \arctan \frac{1}{\sqrt{x}},
\]
and the formula follows.
\end{proof}

\begin{corollary}
Replacing $z$ with the real number $x$ in (\ref{for-main-thm}), multiplying by $\sqrt{x}$ and taking the limit as $x \to 0^{+}$, we see that
\begin{equation}
\lim_{x \to 0^{+}} \sum_{n=1}^{\infty} \frac{\Lambda(n) \sqrt{x}}{\sqrt{n} (1+nx)}=
\lim_{x \to +\infty} \sum_{n=1}^{\infty} \frac{\Lambda(n) \sqrt{x}}{\sqrt{n} (x+n)}=\pi.
\end{equation}
\end{corollary}

\begin{corollary}
For $x$ real sufficiently large the inequality
\begin{equation}\label{ineq}
\left| \sum_{n=1}^{\infty} \frac{\Lambda(n)x}{\pi \sqrt{n} (x+n)}-\sqrt{x}+
\frac{\zeta' (\frac12)}{\pi\zeta (\frac12)} \right| < 10^{-18},
\end{equation}
holds in case that the Riemann Hypothesis is true.
\end{corollary}
\begin{proof}
The function $h(x)$ and the second sum, replacing $z$ with $x$, of the left side of (\ref{h-of-z}) tend to $0$ as $x \to + \infty$. On the other hand, if we assume the RH then the sum over the zeros of zeta in (\ref{h-of-z}) is of order less than $10^{-18}$ for all $x>0$. 
\end{proof}

\begin{theorem}
The following identity
\begin{equation}\label{main2-thm}
\sum_{{\rm Re} \, \tau>0} \frac{\sinh z \tau}{\sinh \pi \tau} - \sum_{n=1}^{\infty} \frac{\Lambda(n)}{2\pi \sqrt{n}} \left(\frac{i e^{iz}}{e^{iz}+n}-\frac{i e^{-iz}}{e^{-iz}+n} \right) = f(z),
\end{equation}
where
\[
f(z)=\sin \frac{z}{2} - \frac{1}{8}\tan \frac{z}{4} -\frac{C+\log 8\pi}{4\pi}\tan \frac{z}{2} - \frac{1}{4\pi \cos \frac{z}{2}} \log \frac{1-\tan \frac{z}{4}}{1+\tan \frac{z}{4}},
\]
holds for $|{\rm Re}(z)|<\pi$. If in addition $|{\rm Im}(z)|<\log 2$, then we have
\begin{equation}\label{main2-cor}
\sum_{{\rm Re} \, \tau>0} \frac{\sinh z \tau}{\sinh \pi \tau} + \frac{1}{\pi} \sum_{n=1}^{\infty} (-1)^n \frac{\zeta'(n+\frac12)}{\zeta(n+\frac12)}\sin nz = f(z).
\end{equation}
\end{theorem}

\begin{proof}
Let 
\[
H(z)=\sqrt{z}-\frac{\zeta'(\frac12)}{\pi \zeta(\frac12)}+h(z).
\]
That is
\[
H(z)=\sqrt{z}-\frac{\zeta'(\frac12)}{\pi \zeta(\frac12)}+\frac{1}{\sqrt{z}(z^2-1)}-\frac{1}{2z-2}+\frac{\log(8\pi)+C}{\pi} \frac{1}{z+1}+\frac{i}{\pi} \frac{\sqrt{z}}{z+1}\log \frac{\sqrt{z}+i}{\sqrt{z}-i}.
\]
From (\ref{for-main-thm}), we see that the function $H(z)$ has the property $H(z)=-H(z^{-1})$. Hence
\begin{align}
H(z)=\frac{H(z)-H(z^{-1})}{2} =& \frac12 \left( \sqrt{z}-\frac{1}{\sqrt{z}} \right)
+ \frac12 \frac{1}{z^2-1} \left( \frac{1}{\sqrt{z}}+z^2 \sqrt{z} \right) - \frac14 \frac{z+1}{z-1} \nonumber \\ &- \frac{\log 8 \pi +C}{2\pi} \frac{z-1}{z+1} + \frac{i}{\pi} \frac{\sqrt{z}}{z+1} 
\log \frac{i\sqrt{z}-1}{i\sqrt{z}+1}+\frac12 \frac{\sqrt{z}}{z+1}. \label{H-of-z}
\end{align}
When $|{\rm Re}(z)|<\pi$ we have $e^{iz} \in \Omega$ so that, we may put $e^{iz}$ instead of $z$ in Theorem \ref{main-thm}. If in addition we multiply by $-i/2$, we get
\[
\sum_{{\rm Re} \, \tau>0} \frac{\sinh z \tau}{\sinh \pi \tau} - \sum_{n=1}^{\infty} \frac{\Lambda(n)}{2 \pi \sqrt{n}} \left(\frac{i e^{iz}}{e^{iz}+n}-\frac{i e^{-iz}}{e^{-iz}+n} \right)=-\frac{i}{2} H(e^{iz}).
\]
From (\ref{H-of-z}), we have
\begin{align}
-\frac{i}{2}H(e^{iz}) =& -\frac{i}{4} \left( e^{iz/2}-e^{-iz/2} \right) 
-\frac{i}{4} \left( \frac{e^{-iz/2}}{e^{2iz}-1}-\frac{e^{iz/2}}{e^{-2iz}-1}  \right)
+\frac{i}{8} \frac{e^{iz}+1}{e^{iz}-1} \nonumber \\
&+\frac{i}{4} \frac{\log 8\pi + C}{\pi} \frac{e^{iz}-1}{e^{iz}+1}+
\frac{1}{2\pi} \frac{e^{iz/2}}{e^{iz}+1} \log \frac{e^{i\frac{z+\pi}{2}}-1}{e^{i \frac{z+\pi}{2}}+1}
-\frac{i}{4} \frac{e^{iz/2}}{e^{iz}+1},
\nonumber
\end{align}
which we can write as
\begin{align}
-\frac{i}{2}H(e^{i z}) =& -\frac{i}{4} \left( e^{i z/2}-e^{-i z/2} \right) 
-\frac{i}{4} \frac{e^{3i z/2}+e^{-3i z/2}}{e^{i z}-e^{-i z}} \nonumber \\
&+\frac{i}{8} \frac{e^{i z/2}+e^{-i z/2}}{e^{i z/2}-e^{-i z/2}} 
+\frac{i}{4} \frac{\log 8\pi + C}{\pi} \frac{e^{i z/2}-e^{-i z/2}}{e^{i z/2}+e^{-i z/2}}  \nonumber \\
&+\frac{1}{2\pi} \frac{1}{e^{i z/2}+e^{-iz/2}} \log \frac{e^{i\frac{z+\pi}{4}}-e^{-i\frac{z+\pi}{4}}}{e^{i \frac{z+\pi}{4}}+e^{-i\frac{z+\pi}{4}}}-\frac{i}{4} \frac{1}{e^{iz/2}+e^{-iz/2}},
\nonumber
\end{align}
which simplifies to
\begin{align}
-\frac{i}{2}H(e^{iz}) =& \frac12 \sin \frac{z}{2}-\frac14 \frac{\cos \left(z+\frac{z}{2}\right)}{\sin z}+\frac18 \cot \frac{z}{2}-\frac{\log 8\pi + C}{4\pi} \tan \frac{z}{2} \nonumber \\ &+ \frac{1}{4\pi} \frac{1}{\cos \frac{z}{2}} \log \left(i \tan \frac{z+\pi}{4} \right)-\frac{i}{8} \frac{1}{\cos \frac{z}{2}}. 
\nonumber
\end{align}
Using elementary trigonometric formulas we arrive at (\ref{main2-thm}). Finally, as
\[ 
\frac{\zeta'(n+\frac12)}{\zeta(n+\frac12)}=-\frac{\Lambda(2)}{\sqrt{2}}  \, 2^{-n}+\mathcal{O}(3^{-n}),
\]
we see that the expression in (\ref{main2-cor}) is convergent for $|{\rm Im} \, z|<\log 2$.
\end{proof}

\begin{example}
Differentiating (\ref{main2-cor}) with respect to $z$ at $z=0$, we get
\begin{equation}
\sum_{n=1}^{\infty} (-1)^n n \frac{\zeta'(n+\frac12)}{\zeta(n+\frac12)} = -\frac38 \log 2 - \frac18 \log \pi + \frac{15}{32} \pi - \frac18 C + \frac18 - \sum_{{\rm Re} \, \tau>0} \frac{\pi \tau}{\sinh \pi \tau}.
\end{equation}
\end{example}

\begin{corollary}
Integrating (\ref{main2-thm}) with respect to $z$, we get 
\begin{equation}\label{formula3}
\sum_{{\rm Re} \, \tau>0} \frac{\cosh z \tau -1}{\tau \sinh \pi \tau}-\sum_{n=1}^{\infty} \frac{\Lambda(n)}{2\pi\sqrt{n}}\left( \log \frac{e^{i z}+n}{1+n} +\log \frac{e^{-i z}+n}{1+n} \right)=g(z),
\end{equation}
where
\[
g(z)=2-2\cos \frac{z}{2}+\frac12 \log \cos \frac{z}{4}+\frac{C+\log 8\pi}{2\pi} \log \cos \frac{z}{2}+\frac{1}{4\pi} \log^2 \left( \frac{1-\tan \frac{z}{4}}{1+\tan \frac{z}{4}} \right),
\]
valid for $|Re(z)|<\pi$. Integrating (\ref{main2-cor}) with respect to $z$, we get
\begin{equation}\label{for3}
\sum_{\tau>0} \frac{\cosh z \tau -1}{\tau \sinh \pi \tau}-\frac{1}{\pi} \sum_{n=1}^{\infty} (-1)^n \frac{\zeta'(n+\frac12)}{\zeta(n+\frac12)} \frac{\cos nz - 1}{n}=g(z),
\end{equation}
valid for $|{\rm Re}(z)|<\pi$ and $|{\rm Im}(z)|<\log 2$.
\end{corollary}

\begin{example}
We use the integrals \cite[eqs. 7.4 \& 7.10]{arsw}
\[
\int_0^{\pi} \log \cos \frac{z}{2} \, dz=-\pi \log 2, \quad \int_0^{\pi} \log \cos \frac{z}{4} \, dz=2 G -\pi \log 2, 
\]
and
\[
\int_0^{\pi} \log^2 \frac{1-\tan \frac{z}{4}}{1+\tan \frac{z}{4}} \, dz = \int_0^1 \frac{4}{1+x^2} \log^2 \frac{1-x}{1+x}dx = 4 \int_0^1 \frac{\log^2 u}{1+u^2}du = 8 \sum_{n=0}^{\infty} \frac{(-1)^n}{(2n+1)^3}=\frac{\pi^3}{4},
\]
(where the penultimate step comes easily from \cite[12.1.5]{boros-moll} for $k=2$), to integrate    (\ref{for3}) between $z=0$ and $z=\pi$. We obtain
\begin{multline}
\sum_{{\rm \, Re \tau} >0} \frac{1}{\tau^2} - \pi \sum_{{\rm \, Re} \, \tau>0} \frac{1}{\tau \sinh \pi \tau}+\sum_{n=1}^{\infty} (-1)^n \frac{1}{n} \frac{\zeta'(n+\frac12)}{\zeta(n+\frac12)} \\ = -4+G+2\pi-\frac{(\pi+C+\log 8 \pi) \log 2}{2}+\frac{\pi^2}{16},
\end{multline}
where $G$ is the Catalan constant. 
\end{example}
\noindent It is interesting to notice the known expression:
\begin{equation}\label{juan}
\sum_{{\rm Re} \, \tau>0} \frac{1}{\tau^2}=-4+G+\frac{\pi^2}{8}+\frac{1}{2} \left( \frac{\zeta''(\frac12)}{\zeta(\frac12)} - \frac{\zeta'(\frac12)^2}{\zeta(\frac12)^2} \right)=0.02310499\dots.
\end{equation}
In addition, from the functional equation of the Riemann zeta function, one can prove that 
\[
\frac{\zeta'(\frac12)}{\zeta(\frac12)}=\frac{C}{2}+\frac{\log 8\pi}{2}+\frac{\pi}{4},
\]
we can check however that it is not possible to use the functional equation to derive an identity for the constant $\zeta''(1/2)/\zeta(1/2)$.

\subsection{Series with the Moebius function}

\begin{theorem}\label{main-moe-thm}
The following identity
\begin{multline}\label{main-moe-for}
\sum_{n=1}^{\infty} \frac{\mu(n)}{2\pi \sqrt{n}} \frac{ie^{iz}}{e^{iz}+n} = \sum_{\gamma} \frac{1}{\zeta'(\frac12+i\gamma)} \frac{e^{-z\gamma}}{\sinh \pi \gamma} \\ + \frac{i}{2 \pi \zeta(\frac12)} - \sum_{n=1}^{\infty} \frac{(-1)^n (2\pi)^{2n}}{(2n)!\zeta(2n+1)} i e^{-\frac{4n+1}{2} i z} + \frac{1}{2\pi} \sum_{n=1}^{\infty}\frac{(-1)^n}{\zeta(\frac12-n)} i e^{-\frac{n}{2}iz},
\end{multline}
holds for $|{\rm Re}(z)|<\pi$ and ${\rm Im}(z)<0$ assuming the Riemann Hypothesis and that all the zeros of zeta are simple.
\end{theorem}
\begin{proof}
Let $\Omega = \mathbb{C}-(-\infty,0]$ (the plane with a cut along the real negative axis). We shall denote by $\log z$ the main branch of the $\log$ function defined on $\Omega$ taking $|\arg(z)|<\pi$. We also denote by $z^s=\exp(s\log(z))$, the usual branch of $z^s$ defined also on $\Omega$. Then, let $I(z)$ be the analytic continuation of the function
\[
I(z)=\int_{-1/2-i\infty}^{-1/2+i\infty} \frac{1}{\zeta(s+\frac12)} \frac{\pi}{\sin \pi s} z^s ds.
\]
along the vertical axis $x=-1/2$. To evaluate $I(z)$, we first consider rectangles of vertices $A=-1/2+iT$, \, $B=R+iT$, \, $C=R-iT$, \, $D=-1/2-iT$, where $R=1/2+\lfloor T \rfloor$. If we choose suitable real numbers $T$ then, by Lemma \ref{lema-reczeta} (which assumes the RH), we see that for $|z|<1$ the integrals, as $T \to \infty$, along the sides $AB$, $BC$ and $CD$ are equal to zero. Hence we can apply the residues theorem, and we obtain
\begin{align}
I &=\sum_{n=0}^{\infty} {\rm res}_{s=n} \left( \frac{1}{\zeta(s+\frac12)} \frac{\pi}{\sin \pi s} z^s \right)+\sum_{\rho} {\rm res}_{s=\rho-\frac12} \left( \frac{1}{\zeta(s+\frac12)} \frac{\pi}{\sin \pi s} z^s \right) \nonumber  \\
&= \sum_{n=0}^{\infty} \frac{(-1)^n}{\zeta(n+\frac12)}z^n+\sum_{\rho} \frac{1}{\zeta'(\rho)}\frac{\pi}{\sin \pi (\rho-\frac12)} z^{\rho-\frac12}. \nonumber
\end{align}
By analytic continuation, we deduce that for $z \in \Omega$ we have
\begin{equation}\label{I-moe-right}
I=\frac{1}{\zeta(\frac12)} + \sum_{n=1}^{\infty} \frac{\mu(n)z}{\sqrt{n}(z+n)}+\sum_{\rho} \frac{1}{\zeta'(\rho)}\frac{\pi}{\sin \pi (\rho-\frac12)} z^{\rho-\frac12}. 
\end{equation}
Integrating along the rectangles of vertices $A=-1/2+iT$, \, $B'=S+iT$, \, $C'=S-iT$, \, $D=-1/2-iT$, where $S=1/2-2 \lfloor T \rfloor$, using Lemma \ref{lema-reczeta} (which assumes the RH), we see that the integrals as $T \to \infty$ along the sides $AB'$, $B'C'$ and $C'D$ are equal to zero. By the residues method, we obtain 
\begin{align}
I &= \sum_{n=1}^{\infty} {\rm res}_{s=-2n-\frac12} \left( \frac{1}{\zeta(s+\frac12)} \frac{\pi}{\sin \pi s} z^s \right) + \sum_{n=1}^{\infty} {\rm res}_{s=-n} \left( \frac{1}{\zeta(s+\frac12)} \frac{\pi}{\sin \pi s} z^s \right) \nonumber \\
&= \sum_{n=1}^{\infty} \frac{\sin 2 \pi n}{\zeta(-2n)}z^{-2n-\frac12} + \sum_{n=1}^{\infty} \frac{(-1)^n}{\zeta(\frac12-n)}  z^{-n}, \label{I-moe-left}
\end{align}
where we understand the expression inside the first sum of (\ref{I-moe-left}) as a limit based on the identity
\[ 
\lim_{s \to -2n} \frac{s+2n}{\zeta(s)} = \lim_{s \to -2n} \frac{\sin{\pi s}}{\pi} \frac{1}{\zeta(s)}.
\] 
From the functional equation of zeta we see that
\[
\sum_{n=1}^{\infty} \frac{\sin 2\pi n}{\zeta(-2n)}=-2\pi \frac{(-1)^n (2\pi)^n}{(2n)!\zeta(2n+1)}.
\]
Identifying (\ref{I-moe-right}) and (\ref{I-moe-left}), multiplying by $i$ and replacing $z$ with $e^{iz}$, we complete the proof of the theorem.
\end{proof}

\section{Asymptotic behaviors} \label{sums-zeros}

The formulas here are related to the work by Cramér \cite{cramer}
\begin{example}
If we let in (\ref{main2-cor}) $z=\pi-1/T$, then we have the following convergent asymptotic expansion as $T \to \infty$:
\begin{align}
\sum_{{\rm Re} \, \tau>0} \exp \left(-\frac{\tau}{T}\right) & = \frac{1}{2\pi}T \log T - \frac{C +\log 2 + \log \pi}{2\pi}T+\frac78+\frac{1}{48\pi}\frac{\log T}{T} \nonumber \\ & + A \frac{1}{T} -\frac{9}{64}\frac{1}{T^2}+\frac{7}{11520\pi}\frac{\log T}{T^3}+\mathcal{O}(\frac{1}{T^3}), \label{asympt-example}
\end{align}
where $A$ is the constant
\[
A=\frac{1}{16} + \frac{4 C -1+16\log 2+4\log \pi}{96\pi}+\frac{1}{\pi}\sum_{n=1}^{\infty} n \frac{\zeta'(n+\frac12)}{\zeta(n+\frac12)}+\sum_{{\rm Re} \, \tau>0} \frac{\tau e^{-\pi \tau}}{\sinh \pi \tau} = -0.759578\dots.
\]
\end{example}

\begin{example}
Differentiating (\ref{asympt-example}) with respect to $T$, we get the following formula:
\begin{align}
\sum_{{\rm Re} \, \tau>0} \tau \exp \left(-\frac{\tau}{T}\right) &=\frac{1}{2\pi} T^2 \log T+ \frac{1-C-\log 2 - \log \pi}{2\pi} T^2-\frac{1}{48\pi}\log T \nonumber \\ &+ \left( \frac{1}{48\pi}-A \right) + \frac{9}{32T}+\mathcal{O}(\frac{1}{T^2}),
\end{align}
as $T \to +\infty$.
\end{example}

\begin{example}
If we differentiate (\ref{asympt-example}) twice with respect to $T$, we obtain
\begin{equation}
\sum_{{\rm Re} \, \tau>0} \tau^2 \exp \left(-\frac{\tau}{T}\right)=\frac{1}{\pi} T^3 \log T + \frac{3-2C-2\log 2\pi}{2\pi} T^3-\frac{1}{48\pi} T -\frac{9}{32} + \mathcal{O}(\frac{1}{T}).
\end{equation}
as $T \to +\infty$.
\end{example}

\section{Representation of arithmetical functions}

In this section we give representations of the functions: Mangoldt $\Lambda(t)$, Moebius $\mu(t)$ and Euler phi $\varphi(t)$, in terms of the non-trivial zeros of the Riemann zeta function. We consider that these functions are defined in $\mathbb{R^{+}}$ rather than in $\mathbb{Z^{+}}$, and we will see that it is natural to define them as $0$ when $t$ is a non-integer.

\subsection{Formulas with the Mangoldt's function}

\begin{theorem}\label{theor-mang}
The identity
\begin{equation}\label{mang-1}
\lim_{x \to \pi^{-}} \left( \frac{1}{4\pi} \frac{\Lambda(t)}{\sqrt{t}} \tan \frac{x}{2} + \sum_{{\rm Re} \, \tau>0} \frac{\sinh x \tau}{\sinh \pi \tau} \cos(\tau \log t) \right) = \frac12 \left( \frac{t+1}{t}-\frac{t}{t^2-1} \right) \sqrt{t},
\end{equation}
holds for $t>1$.
\end{theorem}
\begin{proof}
Replace $z$ with $x-i \log t$ in (\ref{main2-thm}) for $t>1$, take real parts and observe that
\[
{\rm Re} \sum_{n=1}^{\infty} \frac{\Lambda(n)}{2\pi \sqrt{n}} \, \frac{ie^{ix} t}{e^{ix}t+n}=-\sin x \sum_{n=1}^{\infty} \frac{\Lambda(n)}{2\pi \sqrt{n}} \frac{tn}{(e^{ix}t+n)(e^{-ix}t+n)}.
\]
We have to prove that
\[
\lim_{x \to \pi^{-}} \left( -\sin x \sum_{n=1}^{\infty} \frac{\Lambda(n)}{2\pi \sqrt{n}} \frac{tn}{(e^{ix}t+n)(e^{-ix}t+n)}+\frac{\Lambda(t)}{2\pi \sqrt{t}} \, \frac{\sin x}{2(1+\cos x)}\right)=0.
\]
But this identity is evident if $t$ is not the power $p^k$ of a prime $p$. When $t=p^k$ the formula comes observing that the only term that contributes to the sum is $n=p^k$. Using the elementary trigonometric identity
\[
\tan \frac{x}{2}=\frac{\sin x}{1+\cos x},
\]
we complete the proof of the theorem.
\end{proof}

\begin{corollary}\label{coro-mang}
From (\ref{mang-1}) we get the following representation of the Mangoldt function as $x \to \pi^{-}$:
\begin{align}
\Lambda(t) = &-4 \pi \sqrt{t} \cot \frac{x}{2} \sum_{{\rm Re} \, \tau>0} \left( \frac{\sinh x \tau}{\sinh \pi \tau} \cos(\tau \log t) \right) \nonumber \\ &+2\pi \cot \frac{x}{2} \left( t-\frac{1}{t^2-1} \right)+o\left( \cot \frac{x}{2} \right). \label{with-all-zeros}
\end{align}
If we assume the Riemann Hypothesis, then we can replace $\tau$ and ${\rm Re} \, \tau$ with $\gamma$.
\end{corollary}

\noindent We have used Sage  \cite{stein} to write the code below. In it we have taken $x=3.14$ and the first $10000$ zeros of zeta.

\begin{verbatim}

var('t')
z=sage.databases.odlyzko.zeta_zeros()
npi=3.1415926535; x=3.14; b=10000
v(t)=sum([sinh(x*z[j])/sinh(npi*z[j])*cos(log(t)*z[j]) for j in range(b)])
r(t)=-4*npi*sqrt(t)*v(t)*cot(x/2); 
ter(t)=2*npi*(t-1/(t^2-1))*cot(x/2)
plot(r(t)+ter(t),t,2,26)+plot(log(t),t,2,26, color='red')

\end{verbatim}

In Figure \ref{mangoldt} we see the graphic, when we execute the code.
\begin{figure}[H]
\caption{Mangoldt}
\includegraphics[scale=0.75]{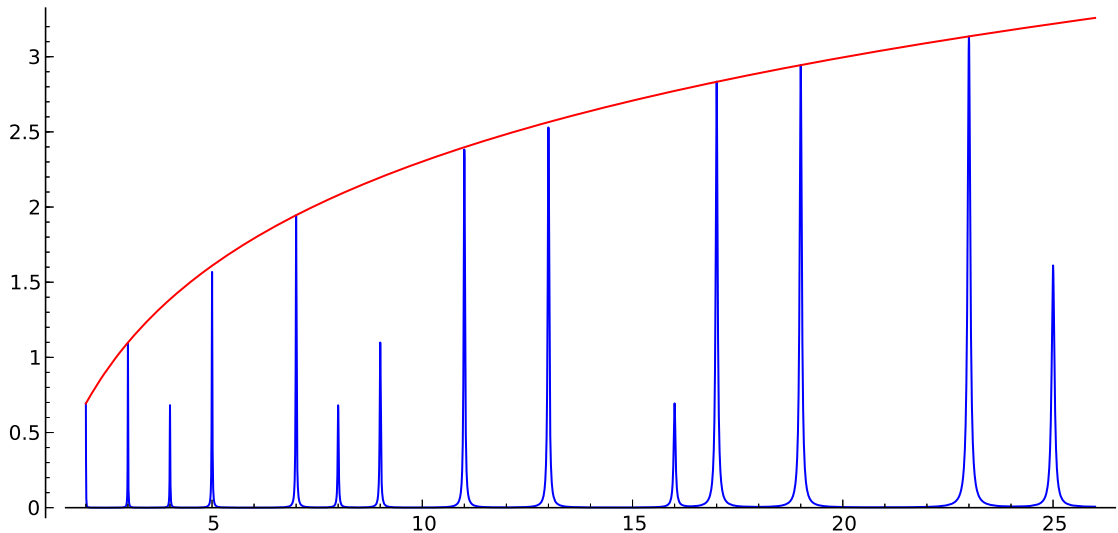}
\label{mangoldt}
\end{figure}

\begin{corollary}
The following formula holds 
\begin{align}
\lim_{x \to \pi^{-}} \sum_{\gamma>0} \frac{\sinh x\gamma}{\sinh \pi \gamma} 
\left( \frac{\log 2}{\sqrt{2}} \cos(\gamma \log t)-\frac{\Lambda(t)}{\sqrt{t}}\cos(\gamma \log 2)
\right) \nonumber \\ = \frac{\log 2}{\sqrt{2}} \left( \frac{\sqrt{t}}{2}-\frac{1}{2(t^2-1)\sqrt{t}} \right)-\frac{\Lambda(t)}{\sqrt{t}} \frac{5\sqrt{2}}{12}. \label{eq-prime}
\end{align}
\end{corollary}
Observe that although we have written the formula assuming RH, we can write it without that assumption if we sum over ${\rm Re} \, \tau>0$ and replace $\gamma$ with $\tau$ in the other places.
\begin{proof}
Substitute $t=2$ in formula (\ref{mang-1}), then use this particular formula together with the general formula (\ref{mang-1}) to eliminate the term involving $\tan(x/2)$. 
\end{proof}
Observe that when $t$ is not a prime nor a power of prime, we have
\[
\lim_{x \to \pi^{-}} \sum_{\gamma>0} \frac{\sinh x\gamma}{\sinh \pi \gamma} 
\cos(\gamma \log t)= \frac{\sqrt{t}}{2}-\frac{1}{2(t^2-1)\sqrt{t}},
\]
and that the limit is infinite otherwise.
\begin{corollary}
The following formula:
\begin{equation}
\lim_{z \to 1^{-}} \sum_{j=0}^{\infty} \cos(\gamma_{j+1} \log t) \, z^j=
\frac{\sqrt{t}}{2}-\frac{1}{2(t^2-1)\sqrt{t}},
\end{equation}
holds whenever $t$ is not a prime nor a power of prime. Otherwise the limit is infinite.
\end{corollary}
\begin{proof}
Just write
\[
\lim_{x \to \pi^{-}} \sum_{j=0}^{\infty} \frac{\sinh(x \gamma_{j+1})}{\sinh(\pi \gamma_{j+1})} \cos(\gamma_{j+1} \log t)=\lim_{x \to \pi^{-}} \lim_{z \to 1^{-}} \sum_{j=0}^{\infty} \frac{\sinh(x \gamma_{j+1})}{\sinh(\pi \gamma_{j+1})} \cos(\gamma_{j+1} \log t) z^j,
\]
and observe that we can commute the limits.
\end{proof}
\begin{conjecture}
The identity
\begin{equation}
\lim_{N \to \infty} \sum_{j=0}^{N} \cos(\gamma_{j+1} \log t) \left( 1-\frac{\log N}{N} \right)^j 
=\frac{\sqrt{t}}{2}-\frac{1}{2(t^2-1)\sqrt{t}},
\end{equation}
is true whenever $t$ is not a prime nor a power of prime. Otherwise the limit is divergent.
\end{conjecture}
In these formulas $\gamma_1$, $\gamma_2$, $\cdots$, denote the imaginary positive parts of the zeros of zeta (a countable set). 
 
\subsection{A known sum over the zeros of zeta}
In next theorem, we reprove in a different way a known formula which is proposed in the book \cite[Exercise 5.11]{ellison}. Its representation has the aspect of a staircase which jumps in the primes, and in the powers of primes. See also the staircase of $\psi(x)$ in \cite{mazur-stein}. 

\begin{theorem}
The identity (which assumes the Riemann Hypothesis)
\begin{equation}\label{staircase}
F(t)-\sum_{\gamma>0} \frac{\sin (\gamma \log t)}{\gamma}=\frac14 \frac{\Lambda(t)}{\sqrt{t}}+\frac12 \sum_{n<t} \frac{\Lambda(n)}{\sqrt{n}},
\end{equation}
where
\[
F(t)=\sqrt{t}-\frac{1}{\sqrt{t}}-\frac12 \arctan \frac{\sqrt{t}-1}{\sqrt{t}+1}-\frac14 \log \frac{\sqrt{t}-1}{\sqrt{t}+1}-\frac14 (\log \pi +3\log 2+C),
\]
holds for $t>1$. 
\end{theorem}
Observe again that although we have written the formula assuming RH, we can write a version of it which does not assume it summing over ${\rm \, Re} \, \tau>0$ and replacing $\gamma$ with $\tau$ in the other places.
\begin{proof}
Replace $z$ with $\pi-i\log t$ in (\ref{formula3}). Then take imaginary parts and observe that each logarithm of a negative real number contributes to the imaginary part with $i \pi$. Finally, observe that for powers of prime numbers we have in addition a logarithm of type $\log (e^{ix}+1)$ with $x \to \pi^{-}$. As the argument of
$1+e^{ix}=(1+\cos x)+i \sin x$ is given by
\[
\tan \omega = \frac{\sin x}{1+\cos x} = \tan \frac{x}{2},
\]
we get that $\omega=\pi/2$. Hence, in case that $t$ be a power of a prime, one of the logarithms gives the extra contribution $i\pi/2$ to the imaginary part.
\end{proof}

\noindent We have used Sage \cite{stein} to represent the graphic of the function in the left side of (\ref{staircase}). Here we show the Sage code that we have written:

\begin{verbatim}

var('t')
z=sage.databases.odlyzko.zeta_zeros()
vv(t)=sum([((sin(log(t)*z[j]))/z[j]) for j in range(10000)])
g1(t)=sqrt(t)-1/sqrt(t)-1/2*arctan((sqrt(t)-1)/(sqrt(t)+1))
g2(t)=-1/8*log((1-sqrt(t))^2/(1+sqrt(t))^2)-1/4*(log(8*pi)+euler_gamma)
rr(t)=g1(t)+g2(t)-vv(t)
plot(rr(t),t,25,55)

\end{verbatim}

\subsection{Formula for the function of Moebius}

\begin{theorem}\label{rep-moebius}
For the Moebius function we have the following representation as $x \to \pi^{-}$ for $t>1$:
\begin{align}
\mu(t)=& 4\pi\sqrt{t} \cot \frac{x}{2} \sum_{\gamma>0} \left[ {\rm Re} \left( \frac{1}{\zeta'(\frac12+i \gamma)} \right) \frac{\sinh x \gamma}{\sinh \pi \gamma}  \cos(\gamma \log t) \right. \nonumber \\ & \left. - {\rm Im} \left( \frac{1}{\zeta'(\frac12+i \gamma)} \right) \frac{\cosh x \gamma}{\sinh \pi \gamma}  \sin(\gamma \log t) \right] \nonumber \\ & + 4\pi \cot \frac{x}{2} 
\sum_{n=1}^{\infty} \frac{(-1)^n (2\pi)^{2n}}{(2n)!\zeta(2n+1)}t^{-2n}+o \left( \cot \frac{x}{2} \right), \label{rep-moe}
\end{align}
if we assume the Riemann Hypothesis and that all the zeros of zeta are simple.
\end{theorem}

\begin{proof}
Replace $z$ with $x-i \log t$ in Theorem \ref{main-moe-thm} and take real parts. Simplify and prove that 
\[
\lim_{x \to \pi^{-}} \left( {\rm Re} \sum_{n=1}^{\infty} \frac{\mu(n)}{\pi \sqrt{n}} \frac{i e^{ix}t}{e^{ix}t+n} + \frac{\mu(t)}{2\pi \sqrt{t}}\tan \frac{x}{2} \right)=0,
\]
similarly as we have done in Theorem \ref{theor-mang}. On the other hand observe that the real parts of $\zeta'(1/2+i\gamma)$ and $\zeta'(1/2-i\gamma)$ are equal, and that their imaginary parts are opposite. Then, continue as in Corollary \ref{coro-mang}.
\end{proof}

\noindent Here is a Sage code which represents the Moebius function using the zeros of zeta. In it we have taken $x=3.14$ and the first $10000$ zeros of zeta.

\begin{verbatim}

def se(j):
    npi=pi.n(digits=12)
    return (-1)^j*(2*npi)^(2*j)/(factorial(2*j)*zeta(2*j+1)).n()
def dz(u):
    return ((zeta(u+10^(-8))-zeta(u))/10^(-8)).n()
z=sage.databases.odlyzko.zeta_zeros()
var('t'); npi=pi.n(digits=12); x=3.14; b=10000
v(t)=sum([sinh(x*z[j])/sinh(npi*z[j])*cos(log(t)*z[j])*
     real(1/dz(0.5+z[j]*I)) for j in range(b)])
w(t)=sum([cosh(x*z[j])/sinh(npi*z[j])*sin(log(t)*z[j])*
     imag(1/dz(0.5+z[j]*I)) for j in range(b)])
r(t)=4*npi*sqrt(t)*(v(t)-w(t))*cot(x/2)
ter(t)=4*npi*sum([se(j)*t^(-2*j) for j in range(1,20)])*cot(x/2)
plot(r(t),1,1,26)+plot(1,1,26,color='red')+plot(-1,1,26,color='red')

\end{verbatim}
We have split two lines at the symbol $*$. However to execute the code with Sage we have to join those lines. Figure \ref{moebius} shows the graphic when we execute the code. Observe that in the last line of the code and hence also in the graphic we have not taken into account the last sum of the formula (\ref{rep-moe}).

\begin{figure}[H]
\caption{Moebius}
\includegraphics[scale=0.75]{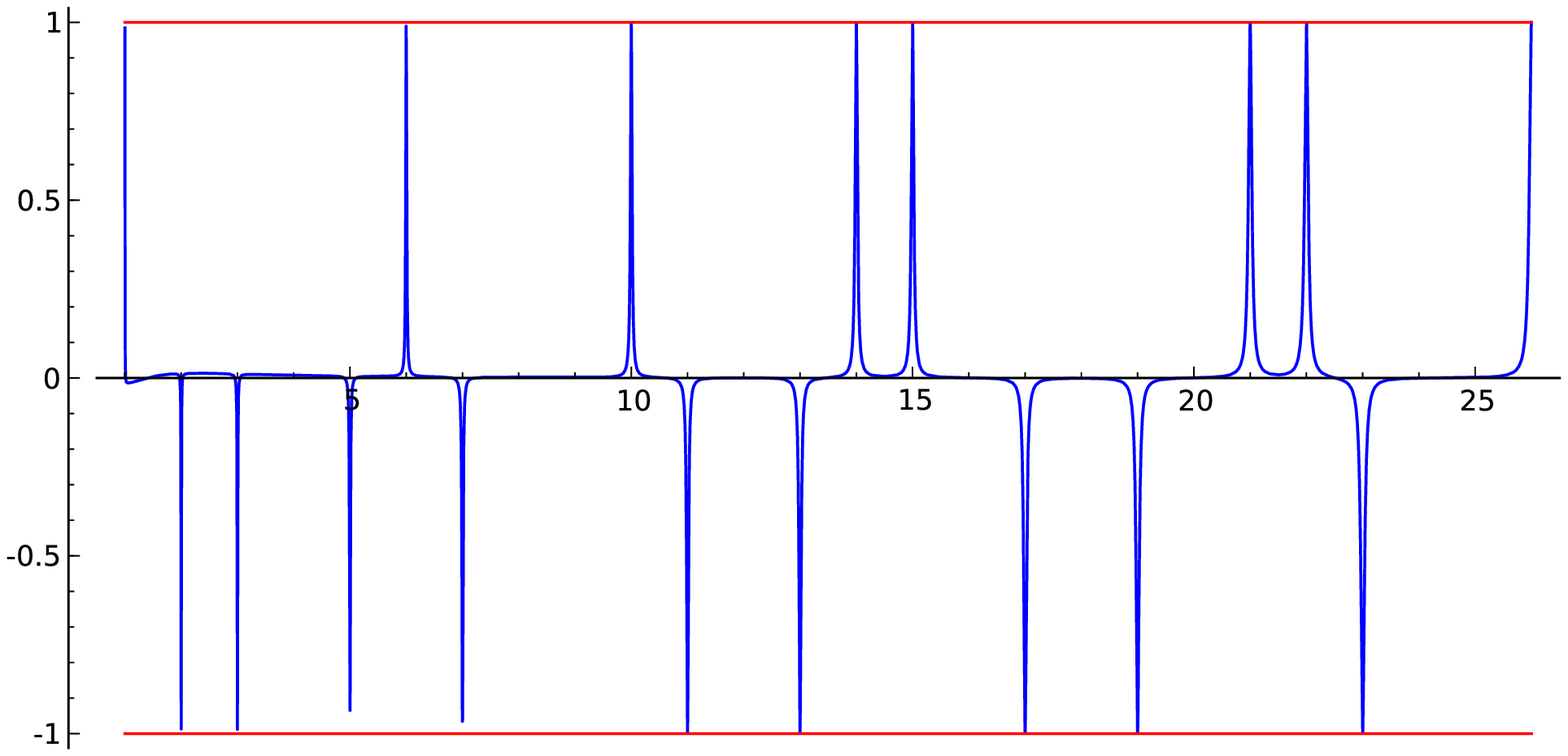}
\label{moebius}
\end{figure}

\subsection{Formulas involving Euler's-Phi function}

\begin{theorem}\label{main-phi-thm}
If we assume the Riemann Hypothesis and that all the zeros of zeta are simple, then the identity
\begin{multline}\label{main-phi-for}
i \frac{\zeta(\frac12)}{2\pi \zeta(\frac32)} e^{iz} - e^{iz} \sum_{n=1}^{\infty} \frac{\varphi(n)}{2 \pi n^{3/2}} \frac{i e^{iz}}{e^{iz}+n}=\sum_{\gamma} \frac{\zeta(-\frac12+i\gamma)}{\zeta'(\frac12+i\gamma)} \frac{e^{-z\gamma}}{\sinh \pi \gamma} + i \frac{\zeta(-\frac12)}{2 \pi \zeta(\frac12)} + i \frac{3}{\pi^2} e^{\frac32 i z} \\ - \frac{1}{2\pi^2} \sum_{n=1}^{\infty} (2n+1) \frac{\zeta(2n+2)}{\zeta(2n+1)} i e^{-\frac{4n+1}{2} i z} + \frac{1}{4 \pi^3} \sum_{n=1}^{\infty} (-1)^n \frac{\zeta(-\frac12-n)}{\zeta(\frac12-n)} i e^{-\frac{n}{2}iz},
\end{multline}
holds for $|{\rm Re}(z)|<\pi$ and ${\rm Im}(z)<0$.
\end{theorem}
\begin{proof}
The left hand side of (\ref{main-phi-for}) comes for the analytic continuation of
\[
\frac{i}{2 \pi} e^{iz} \sum_{n=0}^{\infty} (-1)^n \frac{\zeta(\frac12+n)}{\zeta(\frac32+n)} e^{i z n},
\]
and the proof of the formula (\ref{main-phi-for}) is similar to that of Theorem \ref{main-moe-thm}.
\end{proof}

\begin{theorem}
As $x \to \pi^{-}$, we have the following representation of the Euler-Phi function for $t>1$:
\begin{align}
\varphi(t)=& 4\pi\sqrt{t} \cot \frac{x}{2} \sum_{\gamma>0} \left[ {\rm Re} \left( \frac{\zeta(-\frac12+i \gamma)}{\zeta'(\frac12+i\gamma)} \right) \frac{\sinh x \gamma}{\sinh \pi \gamma}  \cos(\gamma \log t) \right. \nonumber \\ & \left. - {\rm Im} \left( \frac{\zeta(-\frac12+i\gamma)}{\zeta'(\frac12+i \gamma)} \right) \frac{\cosh x \gamma}{\sinh \pi \gamma}  \sin(\gamma \log t) \right] + \frac{12}{\pi} t^2 \cot \frac{x}{2} \nonumber \\ & - \frac{2}{\pi} \cot \frac{x}{2} \sum_{n=1}^{\infty} (2n+1)\frac{\zeta(2n+2)}{\zeta(2n+1)} t^{-2n}+o \left( \cot \frac{x}{2} \right), \label{rep-phi}
\end{align}
which assumes the Riemann Hypothesis and that all the zeros of zeta are simple.
\end{theorem}

\begin{proof}
Replace $z$ with $x-i\log t$ in (\ref{main-phi-for}) and continue as in the proof of Theorem \ref{rep-moebius}.
\end{proof}

\noindent Here is a Sage code for the Euler's phi function. In it we have taken $x=3.14$ and the first $10000$ zeros of zeta.

\begin{verbatim}

def se(j):
    return ((2*j+1)*zeta(2*j+2)/zeta(2*j+1)).n()
def dz(u):
    return ((zeta(u+10^(-8))-zeta(u))/10^(-8))
var('t')
z=sage.databases.odlyzko.zeta_zeros()
npi=pi.n(digits=12); x=3.14; b=10000
v(t)=sum([sinh(x*z[j])/sinh(npi*z[j])*cos(log(t)*z[j])*
     real(zeta(-0.5+z[j]*I)/dz(0.5+z[j]*I)) for j in range(b)])
w(t)=sum([cosh(x*z[j])/sinh(npi*z[j])*sin(log(t)*z[j])*
     imag(zeta(-0.5+z[j]*I)/dz(0.5+z[j]*I)) for j in range(b)])
r(t)=4*npi*sqrt(t)*(v(t)-w(t))*cot(x/2)
ter1(t)=12/npi*t^2*cot(x/2)
ter2(t)=-2/npi*sum([se(j)*t^(-2*j) for j in range(1,20)])*cot(x/2)
plot(r(t)+ter1(t),t,2,26)+plot(t-1,t,2,26,color='red')


\end{verbatim}
Again, we have split two lines at $*$ due to the lack of space. Remember that to execute the code with Sage we have to join those lines. In Figure \ref{eulerphi} we see the graphic when we execute the script. Observe that in the last line of the code and hence also in the graphic we have not taken into account the last sum of the formula (\ref{rep-phi}).

\begin{figure}[H]
\caption{Euler Phi}
\includegraphics[scale=0.75]{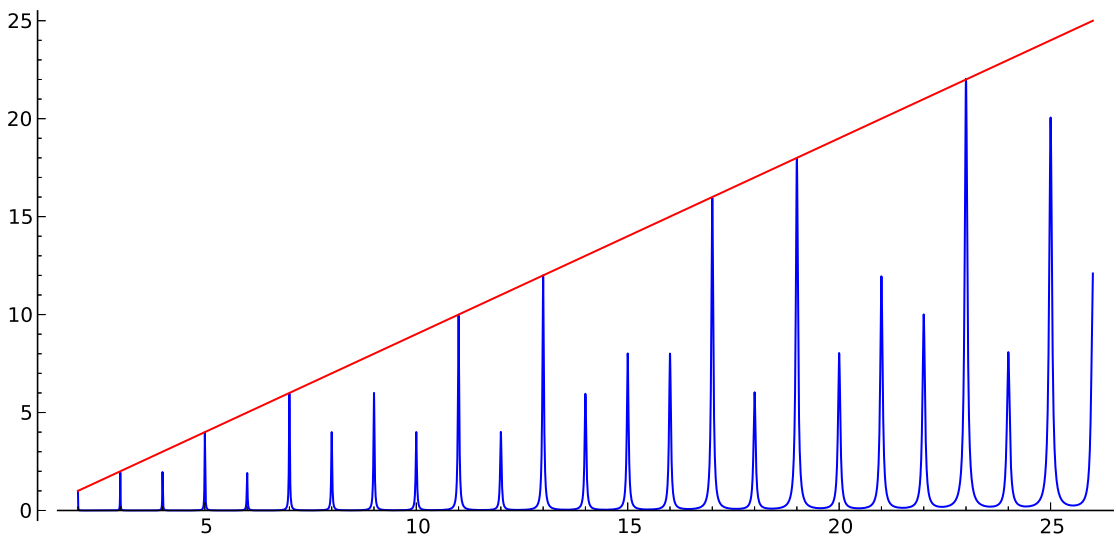}
\label{eulerphi}
\end{figure}

\begin{observation}
{\rm \texttt{sage.databases.odlyzko.zeta\_zeros()}, is a database created by Andrew Odlyzko of the imaginary positive parts of the first $100.000$ zeros of the Riemann zeta function with a precision of $9$ digits.}
\end{observation}

\subsection{Sums over the non-trivial zeros of Dirichlet $L$ functions}
In a similar way one can prove for example that 
\begin{equation}
\chi(t)\Lambda(t)=-4 \pi \sqrt{t} \lim_{x \to \pi^{-}} \left( \cot \frac{x}{2} \sum_{{\rm Re} \, \tau >0} \frac{\sinh x \tau}{\sinh \pi \tau} \cos(\tau \log t) \right),
\end{equation}
where the sum is now over the imaginary positive parts of the non-trivial zeros of $L(s)$ associated to a character $\chi(n)$.

\section*{Acknowledgements}
I am grateful to Jonathan Sondow for some useful comments which have improved the exposition of this paper. Special thanks to Juan Arias de Reyna for noticing to me formula (\ref{juan}) and paper \cite{Kalape}, and for helping with some corrections and interesting comments.

\end{document}